\documentclass[12pt, reqno]{amsart}
\usepackage{amsmath, amsthm, amscd, amsfonts, amssymb, graphicx, color, mathrsfs}
\usepackage[bookmarksnumbered, colorlinks, plainpages]{hyperref}
\usepackage[all]{xy} 
\usepackage{slashed}

\newtheorem{theorem}{Theorem}[section]
\newtheorem{lemma}[theorem]{Lemma}

\theoremstyle{definition}
\newtheorem{definition}[theorem]{Definition}

\theoremstyle{remark}
\newtheorem{remark}[theorem]{Remark}
\numberwithin{equation}{section}

\begin{document}
\setcounter{page}{1}

\title[Fourier multipliers  in Triebel-Lizorkin spaces]{Fourier multipliers for Triebel-Lizorkin spaces on graded Lie groups}

\author[D. Cardona]{Duv\'an Cardona}
\address{
  Duv\'an Cardona S\'anchez:
  \endgraf
  Department of Mathematics: Analysis, Logic and Discrete Mathematics
  \endgraf
  Ghent University, Belgium
  \endgraf
  {\it E-mail address} {\rm duvanc306@gmail.com,\, Duvan.CardonaSanchez@ugent.be }
  }

\author[M. Ruzhansky]{Michael Ruzhansky}
\address{
  Michael Ruzhansky:
  \endgraf
  Department of Mathematics: Analysis, Logic and Discrete Mathematics
  \endgraf
  Ghent University, Belgium
  \endgraf
 and
  \endgraf
  School of Mathematical Sciences
  \endgraf
  Queen Mary University of London
  \endgraf
  United Kingdom
  \endgraf
  {\it E-mail address} {\rm Michael.Ruzhansky@ugent.be}
  }

\thanks{The authors are supported by the FWO Odysseus 1 grant G.0H94.18N: Analysis and Partial Differential Equations. MR is also supported in parts by the EPSRC grant
EP/R003025/2.
}


\begin{abstract} 
In this work we investigate the boundedness of Fourier multipliers on Triebel-Lizorkin spaces associated to positive Rockland operators on a graded Lie group. The found criterion  is expressed in terms of the H\"ormander-Mihlin condition on the global symbol of a Fourier multiplier.
\end{abstract} 

\maketitle

\tableofcontents
\allowdisplaybreaks

\section{Introduction} 
\subsection{Outline}This paper is devoted to  the boundedness of Fourier multipliers of H\"ormander-Mihlin type for Triebel-Lizorkin spaces on graded Lie groups. The boundedness of Fourier multipliers of H\"ormander-Mihlin type has become an indispensable tool in  harmonic analysis. The classical Mihlin multiplier theorem \cite{Mihlin} states that if a function $\sigma\in C^{\infty}(\mathbb{R}^{n}\setminus \{0\}),$ satisfies 
\begin{equation}
    |\partial_{\xi}^{\alpha}\sigma(\xi)|\lesssim_{\alpha} |\xi|^{-|\alpha|},\,\,\,\,|\alpha|\leq [n/2]+1,
\end{equation}then the multiplier $A$ (of the Fourier transform\footnote{ Defined for $f\in C^{\infty}_0(\mathbb{R}^n),$ by $\widehat{f}(\xi):=\int\limits_{\mathbb{R}^n}e^{-2\pi i x\cdot \xi}f(x)dx.$ } on $\mathbb{R}^n$) defined by
\begin{equation}
    Af(x)\equiv T_{\sigma}f(x):=\int\limits_{\mathbb{R}^n}e^{2\pi i x\cdot \xi}\sigma(\xi)\widehat{f}(\xi)d\xi,\,\,f\in C^{\infty}_0(\mathbb{R}^n),
\end{equation}admits a bounded extension on $L^p(\mathbb{R}^n),$ for $1<p<\infty.$ H\"ormander's generalisation of Mihlin's result  in \cite{Hormander1960} guarantees the $L^p$-boundedness of an extension of $A$ under the Sobolev condition
\begin{equation}\label{hormandercondition}
\Vert \sigma\Vert_{l.u.L^2_s}:=\sup_{r>0}\Vert \sigma(r\cdot)\eta(\cdot) \Vert_{L^2_s(\mathbb{R}^n)}<\infty,\,\, \,\,
\end{equation}
where $\eta\in \mathscr{D}(0,\infty),$ $\eta\neq 0,$ and  $s>n/2.$ Later, Calder\'on and  Torchinsky in  \cite{CalderonTorchinsky} extended the H\"ormander-Mihlin theorem to Hardy spaces $H^{p}(\mathbb{R}^n)$ by proving that  $A:H^p({\mathbb{R}}^n)\rightarrow H^p({\mathbb{R}}^n)$ admits a bounded extension provided that \eqref{hormandercondition} holds with $s>n(1/p-1/2)$ and $0<p\leq 1.$ A different proof to the one by  Calder\'on and  Torchinsky was done by Taibleson and Weiss in \cite{TaiblesonWeiss}. The endpoint for the H\"ormander-Mihlin condition in Hardy spaces was found by Baernstein and Sawyer in \cite{BS}. The existence of bounded extensions $A:H^{1}(\mathbb{R}^n)\rightarrow L^{1,2}(\mathbb{R}^n)$  was investigated by Seeger in \cite{Seeger1} and \cite{Seeger2} by considering the Besov condition on the symbol  $\sigma\in B^{2}_{\frac{n}{2},1}(\mathbb{R}^n).$  These estimates were extended to Triebel-Lizorkin spaces  by Seeger in  \cite{Seeger3}. We also refer to the recent paper \cite{Park} of Park  for the generalisation of Seeger's results for Triebel-Lizorkin spaces $F^{r}_{p,q}(\mathbb{R}^n),$ related to the H\"ormander-Mihlin condition.

Because of the numerous applications of the H\"ormander-Mihlin condition on $\mathbb{R}^n$ to the Euclidean harmonic analysis, this condition is also of interest for the harmonic analysis of non-commutative structures, namely, Lie groups and other spaces of homogeneous type (see Coifman and De Guzm\'an \cite{CoifmandeGuzman}). In the context on Lie groups,  the H\"ormander-Mihlin conditon   was extended in \cite{RuzhanskyWirth2015} to arbitrary compact Lie groups\footnote{and also extended in \cite[Section 5]{SubellipticCalculus} to subelliptic Fourier multipliers on compact Lie groups.} generalising the same condition for SU(2) given by Coifman and Weiss in \cite{CoifmanWeiss}. In the case of a graded Lie group $G$, it was proved in  \cite{FR} that the H\"ormander-Mihlin condition for a multiplier $A\equiv T_{\sigma}$ (of the Fourier transform\footnote{which, on a graded Lie group $G,$ with unitary dual $\widehat{G},$ is defined for $f\in C^{\infty}_0(G)$ by $\widehat{f}(\pi):=\int\limits_{G}f(x)\pi(x)^{*}dx,$ at $\pi\in \widehat{G}.$ The Fourier inversion formula is given by  $f(x)=\int\limits_{\widehat{G}}\textnormal{Tr}[\pi(x)\widehat{f}(\pi)]d\pi,$ where $d\pi$ is the Plancherel measure on $\widehat{G}.$ In this way, a Fourier multiplier $A$ is determined by the identity $\widehat{Af}(\pi)=\sigma(\pi)\widehat{f}(\pi),$ for a.e. $\pi\in \widehat{G}.$  } on a graded Lie group $G$), which is defined by
\begin{equation*}
    Af(x)\equiv T_{\sigma}f(x):=\int\limits_{G}\textnormal{Tr}[\pi(x)\sigma(\pi)\widehat{f}(\pi)]d\pi,\,\,f\in C^{\infty}_0(G),
\end{equation*}
implies the existence of a bounded extension of $A$ on $L^p(G),$ for $1<p<\infty.$ This estimate was extended for Hardy spaces in \cite{HongHuRuzhansky2020} generalising the theorem of Calder\'on and Torchinsky to the setting of graded Lie groups. We refer the reader to \cite[Page 39]{CardonaRuzhanskyBesovSpaces} and to \cite{CR} where the H\"ormander-Mihlin condition for right Besov spaces on graded Lie groups was discussed. 

In this work we investigate the H\"ormander-Mihlin  condition for multipliers on Triebel-Lizorkin spaces $F^{r}_{p,q}(G)$ on a graded Lie group $G,$ extending in  Theorem \ref{HMTTL} the estimate of Seeger \cite{Seeger3}  for multipliers in Triebel-Lizorkin spaces $F^{r}_{p,q}(\mathbb{R}^n)$ on $\mathbb{R}^n.$

\subsection{H\"ormander-Mihlin condition for Triebel-Lizorkin spaces}

In order to present our main result, let us introduce the required preliminaries. First we present the H\"ormander-Mihlin condition for  graded Lie groups as introduced in  \cite{FR}.  By a graded Lie group $G,$ we mean   a connected and simply connected nilpotent Lie group $G$ whose Lie algebra $\mathfrak{g} $ may be decomposed as the sum of subspaces 
$$\mathfrak{g}=\mathfrak{g}_{1}\oplus\mathfrak{g}_{2}\oplus \cdots \oplus \mathfrak{g}_{s},$$
such that $[\mathfrak{g}_{i},\mathfrak{g}_{j} ]\subset \mathfrak{g}_{i+j},$ and $ \mathfrak{g}_{i+j}=\{0\}$ if $i+j>s.$ The homogeneous dimension of $Q$ is defined by
\begin{equation*}
    Q:=\sum_{\ell=1}^{s}\ell \cdot \dim(\mathfrak{g}_s).
\end{equation*}
Graded Lie groups include the Euclidean espace $\mathbb{R}^n$, the Heisenberg group $\mathbb{H}^n$  and  any stratified group. Now, consider a positive Rockland operator\footnote{These are linear left invariant homogeneous hypoelliptic partial differential operators, in view of the Helffer and Nourrigat's resolution of the Rockland conjecture in \cite{helffer+nourrigat-79}. Such operators always exist on graded Lie groups and, in fact, the existence of such operators on nilpotent Lie groups does characterise the class of graded Lie groups (c.f. \cite[Section 4.1]{FR2}).} $\mathcal{R}$ of homogeneous degree $\nu>0.$ To define the Triebel-Lizorkin spaces associated to $\mathcal{R},$ let us fix $\eta\in C^{\infty}_0(\mathbb{R}^{+},[0,1]),$ $\eta\neq 0,$ so that $\textnormal{supp}(\eta)\subset [1/2,2],$ and such that 
\begin{equation}
    \sum_{j\in \mathbb{Z}}\eta(2^{-j}\lambda)=1,\,\,\lambda>0.
\end{equation}  Fixing $\psi_0(\lambda):=\sum_{j=-\infty}^{0} \eta_j(\lambda),$ and for $j\geq 1,$ $\psi_j(\lambda):=\eta(2^{-j}\lambda),$ we  have
\begin{equation}
    \sum_{\ell=0}^{\infty}\psi_\ell(\lambda)=1,\,\,\lambda>0,
\end{equation} and one can define the family of operators $\psi_j(\mathcal{R})$ using the functional calculus of $\mathcal{R}.$ Then, for $0<q<\infty,$ and $1<p<\infty,$ the Triebel-Lizorkin space $F^{r}_{p,q}(G)$ consists of the distributions $f\in \mathscr{D}'(G)$ such that
\begin{equation*}
 \Vert f\Vert_{F^{r}_{p,q}(G)}:=   \left\Vert\left(\sum_{\ell=0}^{\infty}2^{\frac{\ell r q}{\nu}}\left|\psi_{\ell}(\mathcal{R})f\right|^{q}\right)^{\frac{1}{q}} \right\Vert_{L^p(G)}<\infty.
\end{equation*}The weak-$F^{r}_{1,q}(G)$ space  is defined by the distributions $f\in \mathscr{D}'(G)$ such that
\begin{equation}
 \Vert f \Vert_{\textrm{weak-}F^{r}_{1,q}(G)}    :=   \sup_{t>0}t\left|\left\{x\in G:\left(\sum_{\ell=0}^\infty2^{\frac{\ell rq}{\nu}}|\psi_\ell(\mathcal{R}) f(x)|^q   \right)^{\frac{1}{q}}>t \right\}\right|<\infty.
\end{equation}
In terms of the Sobolev spaces $L^2_s(\widehat{G})$ on the unitary dual  (see \eqref{SobolevhatG} for details), and of the family of dilations  $\{\sigma(r\cdot \pi)\}_{\pi\in \widehat{G}},$ $r>0,$ of the symbol $\sigma$ of a Fourier multiplier $A\equiv T_\sigma$, the H\"ormander-Mihlin condition takes the form\footnote{Here $\pi(\mathcal{R}),$ a.e. $\pi \in \widehat{G},$ is the symbol of $\mathcal{R},$ characterised by the condition $\widehat{\mathcal{R}f}(\pi)=\pi(\mathcal{R})\widehat{f},$ a.e. $\pi \in \widehat{G},$ and $\eta(\pi(\mathcal{R}))$ is defined by the spectral calculus of Rockland operators (see \cite[Page 178]{FR2}).} (see Subsection \ref{HMLpgradedsection} for details)
\begin{equation}
\Vert \sigma \Vert_{L^2_{s},l.u, R,\eta, \mathcal{R} }:=\sup_{r>0}\Vert\{\sigma(r\cdot \pi)\eta(\pi(\mathcal{R}))   \} \Vert_{L^2_{s}(\widehat{G})}<\infty,
\end{equation}
and 
\begin{equation}
\Vert \sigma \Vert_{L^2_{s},l.u, L,\eta, \mathcal{R} }:=\sup_{r>0}\Vert \{\eta(\pi(\mathcal{R}))\sigma(r\cdot \pi)   \} \Vert_{L^2_{s}(\widehat{G})}<\infty.
\end{equation}

The following theorem is the main result of this work.
\begin{theorem}\label{HMTTL}
Let $G$ be a graded Lie group of homogeneous dimension $Q$. Let $\sigma\in L^{2}(\widehat{G}).$ If 
\begin{equation}\label{LPCRUZFIS}
    \Vert \sigma \Vert_{L^2_{s},l.u, L,\eta, \mathcal{R} },\Vert \sigma \Vert_{L^2_{s},l.u, R,\eta, \mathcal{R} }<\infty,
\end{equation} with $s>Q/2,$ then the corresponding multiplier $A\equiv T_{\sigma}$ extends to a bounded operator from  $F^{r}_{p,q}(G)$ into $F^{r}_{p,q}(G)$  for all $1<p,q<\infty,$ and all $r\in \mathbb{R}.$ Moreover
\begin{equation}
\Vert T_{\sigma} \Vert_{ \mathscr{L}(F^{r}_{p,q}(G))    }\leq C  \max\{ \Vert \sigma \Vert_{L^2_{s},l.u, L,\eta, \mathcal{R} },\Vert \sigma \Vert_{L^2_{s},l.u, R,\eta, \mathcal{R} } \},
\end{equation}and for $p=1,$ $A\equiv T_{\sigma}$ admits a bounded extension from $F^{r}_{1,q}$ into $\textrm{weak-}F^{r}_{1,q}(G),$ and
\begin{equation}
\Vert T_\sigma  \Vert_{\mathscr{L}\left(F^{r}_{1,q}(G),\,\textrm{weak-}F^{r}_{1,q}(G)\right)}    \leq C \max\{ \Vert \sigma \Vert_{L^2_{s},l.u, L,\eta, \mathcal{R} },\Vert \sigma \Vert_{L^2_{s},l.u, R,\eta, \mathcal{R} } \},
\end{equation}for any $1<q<\infty.$
\end{theorem}
Now, we discuss briefly our result.
\begin{remark}
In the case $G=\mathbb{R}^n,$ and taking $\mathcal{R}=(-\Delta_x)^{\frac{1}{2}},$ where $\Delta_x$ is the negative Laplacian on $\mathbb{R}^n,$ Theorem \ref{HMTTL} recovers the H\"ormander-Mihlin theorem in Seeger \cite{Seeger3} for Triebel-Lizorkin spaces on $\mathbb{R}^n$.
Also, in view of the Littlewood-Paley theorem in \cite{CardonaRuzhanskyBesovSpaces}, $F^{0}_{p,2}(G)=L^p(G),$ for all $1<p<\infty,$ Theorem \ref{HMTTL} recovers the $L^p(G)$-H\"ormander Mihlin theorem in \cite{FR2}. 

\end{remark}
\begin{remark}
The H\"ormander-Mihlin theorem has been extended by several authors to spectral multipliers of Laplacian and sub-Laplacians, and settings that go beyond the Euclidean case.  The literature is so broad that it is impossible to provide complete list here. We refer the reader to \cite{alexo,CardonaDelgadoRuzhansky2019,Sikora} and to the extensive list of references therein.
\end{remark}
\section{Preliminaries}\label{Preliminaries}

\noindent In this section, we recall  some preliminaries  on graded and homogeneous Lie groups $G$. The unitary dual of these groups will be denoted by $\widehat{G}$. We  also present the notion of Rockland operators and Sobolev spaces on $G$ and on the unitary dual $\widehat{G}$ by following \cite{FR}, to which we refer for further details on constructions presented in this section. For the general aspects of the harmonic analysis on nilpotent Lie groups we refer the reader to \cite{FR2,FE}.

\subsection{Dilations on a graded Lie group} Let $G$ be a graded Lie group. This means that $G$ is a connected and simply connected nilpotent Lie group whose Lie algebra $\mathfrak{g}$ may be decomposed as the sum of subspaces $\mathfrak{g}=\mathfrak{g}_{1}\oplus\mathfrak{g}_{2}\oplus \cdots \oplus \mathfrak{g}_{s}$ such that $[\mathfrak{g}_{i},\mathfrak{g}_{j} ]\subset \mathfrak{g}_{i+j},$ and $ \mathfrak{g}_{i+j}=\{0\}$ if $i+j>s.$ This implies that the group $G$ is nilpotent because the sequence
$$ \mathfrak{g}_{(1)}:=\mathfrak{g},\,\,\,\,\mathfrak{g}_{(n)}:=[\mathfrak{g},\mathfrak{g}_{(n-1)}] $$
defined inductively terminates at $\{0\}$ in a finite number of steps. Examples of such groups are the Heisenberg group $\mathbb{H}^n$ and more generally any stratified groups where the Lie algebra $ \mathfrak{g}$ is generated by $\mathfrak{g}_{1}$. The exponential mapping from $\mathfrak{g}$ to $G$ is  a diffeomorphism, then, we can identify $G$ with $\mathbb{R}^n$ or $\mathfrak{g}_{1}\times \mathfrak{g}_{2}\times \cdots \times \mathfrak{g}_{s}$ as  manifolds. Consequently we denote by $\mathscr{S}(G)$ the Schwartz space of functions on $G,$ by considering the identification $G\equiv \mathbb{R}^n.$ Here, $n$ is the topological dimension of $G,$ $n=n_{1}+\cdots +n_{s},$ where $n_{k}=\mbox{dim}\mathfrak{g}_{k}.$ A family of dilations $D_{r},$ $r>0,$ on a Lie algebra $\mathfrak{g}$ is a family of linear mappings from $\mathfrak{g}$ to itself satisfying the following two conditions:
\begin{itemize}
\item For every $r>0,$ $D_{r}$ is a map of the form
$$ D_{r}=\textnormal{Exp}(\ln(r)A) $$
for some diagonalisable linear operator $A$ on $\mathfrak{g}.$
\item $\forall X,Y\in \mathfrak{g}, $ and $r>0,$ $[D_{r}X, D_{r}Y]=D_{r}[X,Y].$ 
\end{itemize}
We call  the eigenvalues of $A,$ $\nu_1,\nu_2,\cdots,\nu_n,$ the dilations weights or weights of $G$. A homogeneous Lie group is a connected simply connected Lie group whose Lie algebra $\mathfrak{g}$ is equipped with a family of dilations $D_{r}.$ In  such case, and with the notation above,  the homogeneous dimension of $G$ is given by  $$ Q=\textnormal{Tr}(A)=\sum_{l=1}^{s}l\cdot\dim \mathfrak{g}_l.  $$
We can transport dilations $D_{r}$ of the Lie algebra $\mathfrak{g}$ to the group by considering the family of maps
$$ \exp_{G}\circ D_{r} \circ \exp_{G}^{-1},\,\, r>0, $$
where $\exp_{G}:\mathfrak{g}\rightarrow G$ is the usual exponential function associated to the Lie group $G.$ We denote  this family of dilations also by $D_{r}$ and we refer to them as dilations on the group. If we write $rx=D_{r}(x),$ $x\in G,$ $r>0,$ then a relation on the homogeneous structure of $G$ and the Haar measure $dx$ on $G$ is given by $$ \int\limits_{G}(f\circ D_{r})(x)dx=r^{-Q}\int\limits_{G}f(x)dx. $$

\subsection{The unitary dual and the Plancherel theorem} We will always equip a graded Lie group with the Haar measure $dx.$ For simplicity, we will write $L^p(G)$ for $L^p(G, dx).$ We denote by $\widehat{G}$  the unitary dual of $G,$ that is the set of equivalence classes of unitary, irreducible, strongly continuous representations of $G$ acting in separable Hilbert spaces. The unitary dual can be equipped with the Plancherel measure  $d\mu.$ So, the Fourier transform of every  function $\varphi\in \mathscr{S}(G)$ at $\pi\in \widehat{G}$ is defined by 
 $$  (\mathscr{F}_{G}\varphi)(\pi)\equiv\widehat{\varphi}(\pi)=\int\limits_{G}\varphi(x)\pi(x)^*dx, $$
 and the corresponding Fourier inversion formula is given by
 $$  \varphi(x)=\int_{ \widehat{G}}\mathrm{Tr}(\pi(x)\widehat{\varphi}(\pi))d\mu(\pi).$$
In this case, we have the Plancherel identity
$$ \Vert \varphi \Vert_{L^2(G)}= \left(\int_{ \widehat{G}}\mathrm{Tr}(\widehat{\varphi}(\pi)\widehat{\varphi}(\pi)^*)d\mu(\pi) \right)^{\frac{1}{2}}=\Vert  \widehat{\varphi}\Vert_{ L^2(\widehat{G} ) } .$$
We also denote $\Vert\widehat{\varphi} \Vert^2_{\textnormal{HS}}=\textnormal{Tr}(\widehat{\varphi}(\pi)\widehat{\varphi}(\pi)^*)$ the Hilbert-Schmidt norm of operators.  A  Fourier multiplier  is formally defined by

\begin{equation}\label{mul}T_{\sigma}u(x)=\int_{ \widehat{G}}\mathrm{Tr}(\pi(x)\sigma(\pi)\widehat{f}(\pi))d\mu(\pi),
\end{equation}

\noindent where the symbol $\sigma(\pi)$  is defined on the unitary dual $\widehat{G}$ of $G.$ For a rather comprehensive treatment of this quantization we refer to \cite{FR2}. 
\subsection{Homogeneous linear operators and Rockland operators} A linear operator $T:\mathscr{D}(G)\rightarrow \mathscr{D}'(G)$ is homogeneous of  degree $\nu\in \mathbb{C}$ if for every $r>0$ 
\begin{equation}
T(f\circ D_{r})=r^{\nu}(Tf)\circ D_{r}
\end{equation}
holds for every $f\in \mathscr{D}(G). $
If for every representation $\pi\in\widehat{G},$ $\pi:G\rightarrow U(\mathcal{H}_{\pi}),$ we denote by $\mathcal{H}_{\pi}^{\infty}$ the set of smooth vectors, that is, the space of elements $v\in \mathcal{H}_{\pi}$ such that the function $x\mapsto \pi(x)v,$ $x\in \widehat{G}$ is smooth,  a Rockland operator is a left-invariant differential operator $\mathcal{R}$ which is homogeneous of positive degree $\nu=\nu_{\mathcal{R}}$ and such that, for every unitary irreducible non-trivial representation $\pi\in \widehat{G},$ $\pi(\mathcal{R})$ is injective on $\mathcal{H}_{\pi}^{\infty};$ $\sigma_{\mathcal{R}}(\pi)=\pi(\mathcal{R})$ is the symbol associated to $\mathcal{R}.$ It coincides with the infinitesimal representation of $\mathcal{R}$ as an element of the universal enveloping algebra. It can be shown that a Lie group $G$ is graded if and only if there exists a differential Rockland operator on $G.$ If the Rockland operator is formally self-adjoint, then $\mathcal{R}$ and $\pi(\mathcal{R})$ admit self-adjoint extensions on $L^{2}(G)$ and $\mathcal{H}_{\pi},$ respectively. Now if we preserve the same notation for their self-adjoint
extensions and we denote by $E$ and $E_{\pi}$  their spectral measures, by functional calculus we have
$$ \mathcal{R}=\int\limits_{-\infty}^{\infty}\lambda dE(\lambda),\,\,\,\textnormal{and}\,\,\,\pi(\mathcal{R})=\int\limits_{-\infty}^{\infty}\lambda dE_{\pi}(\lambda). $$
We now recall a lemma on dilations on the unitary dual $\widehat{G},$ which will be useful in our analysis of spectral multipliers.   For the proof, see Lemma 4.3 of \cite{FR}.
\begin{lemma}\label{dilationsrepre}
For every $\pi\in \widehat{G}$ let us define $$D_{r}(\pi)\equiv r\cdot \pi:=\pi^{(r)}$$ by $D_{r}(\pi)(x)=\pi(r\cdot x)$ for every $r>0$ and $x\in G.$ Then, if $f\in L^{\infty}(\mathbb{R})$ then $f(\pi^{(r)}(\mathcal{R}))=f({r^{\nu}\pi(\mathcal{R})}).$
\end{lemma}

We refer to \cite[Chapter 4]{FR2} and references therein  for an exposition of further properties of Rockland operators and their history, and to ter Elst and Robinson \cite{TElst+Robinson} for their spectral properties.

\subsection{H\"ormander-Mihlin multipliers on $L^p(G)$}\label{HMLpgradedsection}
To define Sobolev spaces, we choose  a positive left-invariant Rockland operator $\mathcal{R}$ of homogeneous degree $\nu>0$.  With notations above one defines  Sobolev spaces as follows (c.f \cite{FR2}).
\begin{definition}
Let $r\in \mathbb{R},$ the homogeneous Sobolev space $\dot{L}^p_r(G)$ consists of those $f\in \mathcal{D}'(G)$ satisfying
\begin{equation}
\Vert f\Vert_{\dot{L}^p_r(G)}:=\Vert \mathcal{R}^{\frac{r}{\nu}}f \Vert_{L^p(G)}<\infty.
\end{equation}
Analogously, the inhomogeneous Sobolev space ${H}^{r,p}(G)$ consists of those distributions $f\in \mathcal{D}'(G)$ satisfying 
\begin{equation}
\Vert f\Vert_{L^p_r(G)}:=\Vert (I+ \mathcal{R})^{\frac{r}{\nu}}f \Vert_{L^p(G)}<\infty.
\end{equation}
\end{definition}
By using a quasi-norm  $|\cdot| $ on $G$ we can introduce  for every $r\geq 0,$ the inhomogeneous Sobolev space of order $r$ on $\widehat{G},$ $L^2_{r}(\widehat{G})$ which is defined by
\begin{equation}\label{SobolevhatG}
    L^2_{r}(\widehat{G})=\mathscr{F}_{G}(L^{2}(G, (1+|\cdot|^{2})^{\frac{r}{2}}dx))
\end{equation}
 where $\mathscr{F}_{G}$ is the Fourier transform on the group $G.$  In a similar way, for $r\geq0$ the homogeneous Sobolev space $\dot{L}^2_{r}(\widehat{G})$ is defined by
$$\dot{L}^2_{r}(\widehat{G})=\mathscr{F}_{G}(L^{2}(G, |\cdot|^{r}dx)).$$
As usual if $r=0$ we denote $L^{2}(\widehat{G})=\dot{H}^{0}(\widehat{G})=H^{0}(\widehat{G}).$ Characterisations of Sobolev spaces on $G$ and on the unitary dual $\widehat{G}$ in terms of homogeneous norms on $G$ can be found in \cite{FR} and \cite{FR2}, respectively. 

Finally we present the H\"ormander-Mihlin theorem for graded nilpotent Lie groups.  The formulation of such result  requires a local notion of Sobolev space on the dual space $\widehat{G}.$  We introduce this as follows. Let $s\geq 0,$  we say that the field $\sigma=\{ \sigma(\pi):\pi\in\widehat{G}\}$ is locally uniformly in right-$L^2_{s}(\widehat{G})$ (resp. left-$L^2_{s}(\widehat{G})$)  if there exists a positive Rockland operator $\mathcal{R}$ and a function $\eta\in \mathcal{D}(G),$ $\eta\neq 0,$ satisfying
\begin{equation}
\Vert \sigma \Vert_{L^2_{s},l.u, R,\eta, \mathcal{R} }:=\sup_{r>0}\Vert\{\sigma(r\cdot \pi)\eta(\pi(\mathcal{R}))   \} \Vert_{L^2_{s}(\widehat{G})}<\infty,
\end{equation}
respectively,
\begin{equation}
\Vert \sigma \Vert_{L^2_{s},l.u, L,\eta, \mathcal{R} }:=\sup_{r>0}\Vert \{\eta(\pi(\mathcal{R}))\sigma(r\cdot \pi)   \} \Vert_{L^2_{s}(\widehat{G})}<\infty.
\end{equation}
It important to mention that if $\phi\neq 0,$ is another function in $\mathcal{D}(0,\infty)$ then (see \cite{FR})
\begin{equation}\label{independencelu}
\Vert \sigma \Vert_{L^2_{s},l.u, R,\eta, \mathcal{R} }\asymp \Vert \sigma \Vert_{L^2_{s},l.u, R,\phi, \mathcal{R} },\textnormal{      and      } \Vert \sigma \Vert_{L^2_{s},l.u, L,\eta, \mathcal{R} }\asymp \Vert \sigma \Vert_{L^2_{s},l.u, L,\phi, \mathcal{R} }.
\end{equation}
The following lemma shows how Sobolev spaces on the unitary dual interact with the family of dilations.
  \begin{lemma}\label{lemmaLP}
Let $\sigma\in L^{2}(\widehat{G}).$ If $r>0$ and $s\geq 0$ then 
\begin{equation}
\Vert \sigma \circ D_{r}\Vert_{\dot{L}^2_{s}(\widehat{G})}= r^{s-\frac{Q}{2}}\Vert \sigma \Vert_{\dot{L}^2_{s}(\widehat{G})}.
\end{equation}
This implies that $\sigma \in \dot{L}^2_{s}(\widehat{G})$ if only if for every $r>0,$ $\sigma\circ D_{r}\in \dot{L}^2_{s}(\widehat{G}).$ Also, if $\mathcal{R},\mathcal{S}$ are positive Rockland operators and $\eta,\zeta\in \mathcal{D}(0,\infty),$ $\eta,\zeta\neq 0,$ then there exists $C>0$ such that
\begin{equation}
\Vert \sigma \Vert_{L^2_s,l.u,L,\zeta,\mathcal{S}}\leq C \Vert \sigma \Vert_{L^2_s,l.u,L,\eta,\mathcal{R}}
\end{equation}
and 
\begin{equation}
\Vert \sigma \Vert_{L^2_s,l.u,R,\zeta,\mathcal{S}}\leq C \Vert \sigma \Vert_{L^2_s,l.u,R,\eta,\mathcal{R}}.
\end{equation}
\end{lemma}
\begin{proof} By Lemma 2.1 or Lemma 4.3 of \cite{FR} we have
\begin{align*}\Vert \sigma\circ D_{r} \Vert_{\dot{L}^2_{s}(\widehat{G})} &=\Vert |\cdot|^{s} \mathscr{F}_{G}^{-1}(\sigma \circ D_{r}) \Vert_{L^{2}({G})}=\Vert |\cdot|^{s} r^{-Q}\mathscr{F}_{G}^{-1}(\sigma )(r^{-1}\cdot) \Vert_{L^{2}({G})}\\
&=r^{-\frac{Q}{2}}\Vert |r\cdot|^{s} \mathscr{F}_{G}^{-1}(\sigma ) \Vert_{L^{2}({G})}\\
&=r^{s-\frac{Q}{2}}\Vert \sigma\Vert_{\dot{L}^2_{s}(\widehat{G})}.
\end{align*}
With the equality above, it is clear that  $\sigma \in \dot{L}^2_{s}(\widehat{G})$ if only if for every $r>0,$ $\sigma\circ D_{r}\in \dot{L}^2_{s}(\widehat{G}).$  The second part of the Lemma  has been shown in Proposition 4.6 of \cite{FR}.
\end{proof}

Now, we state the H\"ormander-Mihlin theorem on the graded nilpotent Lie group $G$ (c.f. Theorem 4.11 of \cite{FR}):
\begin{theorem}\cite[$L^p$-H\"ormander-Mihlin Theorem]{FR}.\label{HMT}
Let $G$ be a graded Lie group. Let $\sigma\in L^{2}(\widehat{G}).$ If 
\begin{equation}\label{LPCRUZFIS2}
    \Vert \sigma \Vert_{L^2_{s},l.u, L,\eta, \mathcal{R} },\Vert \sigma \Vert_{L^2_{s},l.u, R,\eta, \mathcal{R} }<\infty,
\end{equation} with $s>Q/2,$ then the corresponding multiplier $A\equiv T_{\sigma}$ extends to a bounded operator on $L^{p}(G)$ for all $1<p<\infty,$ and for $p=1,$ $T_{\sigma}$  is of weak (1,1) type. Moreover
\begin{equation}
\Vert T_{\sigma} \Vert_{ \mathcal{L}(L^1(G),\,\mathcal{L}^{1,\infty}(G))    },\,\Vert T_{\sigma} \Vert_{ \mathcal{L}(L^p(G))    }\leq C  \max\{ \Vert \sigma \Vert_{L^2_{s},l.u, L,\eta, \mathcal{R} },\Vert \sigma \Vert_{L^2_{s},l.u, R,\eta, \mathcal{R} } \}.
\end{equation}
\end{theorem}
In the proof of Theorem \ref{HMTTL} we will use that every dyadic decomposition of a Fourier multiplier satisfying the H\"ormander-Mihlin condition has a Calder\'on-Zygmund kernel and we will make use of the estimates proved in \cite{FR}  for this family of kernels. So, in the following remark we record the Calder\'on-Zygmund estimates in    \cite{FR}.
\begin{remark}[On the proof of the $L^p$-H\"ormander-Mihlin Theorem]\label{theremarkof} Let us describe the fundamental steps of the proof of the $L^p$-H\"ormander-Mihlin theorem (c.f. Theorem 4.11 of \cite{FR}) on graded Lie groups. For this, we follow \cite[Page 19]{FR}. Let us fix $\eta\in C^{\infty}_0(\mathbb{R}^{+},[0,1]),$ $\eta\neq 0,$ so that $\textnormal{supp}(\eta)\subset [1/2,2],$ and such that 
\begin{equation}
    \sum_{j\in \mathbb{Z}}\eta(2^{-j}\lambda)=1,\,\,\lambda>0.
\end{equation}
By defining $\psi_0(\lambda):=\sum_{j=-\infty}^{0} \eta_j(\lambda),$ and for $j\geq 1,$ $\psi_j(\lambda):=\eta(2^{-j}\lambda),$ we obviously have
\begin{equation}
    \sum_{\ell=0}^{\infty}\psi_\ell(\lambda)=1,\,\,\lambda>0.
\end{equation}Then, for a.e. $\pi\in \widehat{G},$ $ \sum_{\ell=0}^{\infty}\psi_\ell(\mathcal{R})$ converges towards the identity in the strong topology of the norm in $L^2(G).$ By decomposing 
\begin{equation}
    T_\sigma=\sum_{j\geq 0}T_{j},\,\,\, T_{j}:=T_\sigma \psi_{j}(\mathcal{R}),
\end{equation} and using that the right-convolution  kernels  of the family $T_j,$ $k_{j}$ summed on $j,$ provide the  distributional kernel of $T,$ $k=\sum_{j}k_{j},$ which agrees with a locally integrable function on $G\setminus \{0\},$ such that, for every $c>0,$
\begin{equation}\label{ThLiWeak}
    \mathscr{I}_\ell:=\sup_{z\in G}\int\limits_{ |x|>4c|z|}|2^{-\ell Q}\kappa_{\ell}(2^{-\ell}\cdot z^{-1} x)  - 2^{-\ell Q}\kappa_{\ell}(2^{-\ell }\cdot x)|dx,
\end{equation} satisfies, $\mathscr{I}_\ell\lesssim    2^{-\ell \varepsilon_0}\max\{\Vert \sigma \Vert_{L^2_{s},l.u, L,\eta, \mathcal{R} },\Vert \sigma \Vert_{L^2_{s},l.u, R,\eta, \mathcal{R} }\},$ for some $\varepsilon_0>0,$ depending only of $c>0.$
The proof  in \cite[Page 19]{FR} shows that  $$\Vert T_j\Vert_{\mathscr{B}(L^p(G))}\leq \mathscr{I}_j   \max\{\Vert \sigma \Vert_{L^2_{s},l.u, L,\eta, \mathcal{R} },\Vert \sigma \Vert_{L^2_{s},l.u, R,\eta, \mathcal{R} }\}$$ and consequently 
$$\Vert T\Vert_{\mathscr{B}(L^p(G))}\lesssim \sum_j  2^{-j \varepsilon_0}\max\{\Vert \sigma \Vert_{L^2_{s},l.u, L,\eta, \mathcal{R} },\Vert \sigma \Vert_{L^2_{s},l.u, R,\eta, \mathcal{R} }\}, $$
proving the $L^p(G)$-boundedness of $T_\sigma.$ 
 
\end{remark}

About, the Littlewood-Paley decomposition $\{\psi_\ell\}_{\ell=0}^{\infty},$ introduced in Remark \ref{theremarkof}, the following estimates were proved in \cite{CardonaRuzhanskyBesovSpaces}. The following result is the Littlewood-Paley theorem.
\begin{theorem}\cite[Littlewood Paley Theorem]{CardonaRuzhanskyBesovSpaces}.\label{LPT}
Let  $1<p<\infty$ and let $G$ be a graded Lie group. If $\mathcal{R}$ is a positive Rockland operator then there exist  constants $0<c_p,C_{p}<\infty$ depending only on $p$ and $\psi_0$ such that
\begin{equation}\label{LPTequ}
c_p\Vert f\Vert_{L^{p}(G)}\leq \left\Vert \left(\sum_{\ell=0}^{\infty} |\psi_{\ell}(\mathcal{R})f|^{2}    \right)^{\frac{1}{2}}\right\Vert_{L^{p}(G)}\leq  C_{p}\Vert f\Vert_{L^{p}},
\end{equation}
holds for every $f\in L^{p}(G).$ Moreover, for $p=1,$  there exists a constant $C>0$ independent of $f\in L^1(G)$ and $t>0,$ such that 
\begin{equation}\label{weak(1,1)inequality}
    \left|\left\{x\in G:\left(\sum_{\ell=0}^\infty|\psi_\ell(\mathcal{R}) f(x)|^2   \right)^{\frac{1}{2}}>t \right\}\right|\leq \frac{C}{t}\Vert f\Vert_{L^1(G)}.
\end{equation}
\end{theorem}
The action of dyadic decompositions on vector-valued functions is considered in the next theorem.
\begin{theorem}\cite[Vector-valued inequality for dyadic decompositions]{CardonaRuzhanskyBesovSpaces}.\label{Theoremr}
Let  $1<p,r<\infty$ and let $G$ be a graded Lie group. If $\mathcal{R}$ is a positive Rockland operator then there exist  constants $C_p>0$ depending only on $p$ and $\psi_0,$ such that
\begin{equation}\label{1pinfty}
 \left\Vert \left(\sum_{\ell=0}^{\infty} |\psi_{\ell}(\mathcal{R})f_\ell|^{r}    \right)^{\frac{1}{r}}\right\Vert_{L^{p}(G)}\leq C_p \left\Vert \left(\sum_{\ell=0}^{\infty} |f_\ell(x)|^{r}    \right)^{\frac{1}{r}}\right\Vert_{L^{p}(G)}=: C_{p}\Vert \{f_\ell\}\Vert_{L^{p}(G,\ell^r(\mathbb{N}_0^n))}.
\end{equation}Moreover, for $p=1,$  there exists a constant $C>0$ independent of $\{f_\ell\}\in L^1(G,\ell^r(\mathbb{N}_0))$ and $t>0,$ such that 
\begin{equation}\label{weakvectorvaluedinequality}
    \left|\left\{x\in G:\left(\sum_{\ell=0}^\infty|\psi_\ell(\mathcal{R}) f_\ell(x)|^r   \right)^{\frac{1}{r}}>t \right\}\right|\leq \frac{C}{t}\Vert \{f_\ell\} \Vert_{L^1(G,\ell^r(\mathbb{N}_0)}.
\end{equation}
\end{theorem}

\section{Triebel-Lizorkin spaces on graded Lie groups} In this section, Triebel-Lizorkin spaces  on graded Lie groups are introduced. They can be defined by using positive Rockland operators. As in the introduction, let us fix $\eta\in C^{\infty}_0(\mathbb{R}^{+},[0,1]),$ $\eta\neq 0,$ so that $\textnormal{supp}(\eta)\subset [1/2,2],$ and such that 
\begin{equation}
    \sum_{j\in \mathbb{Z}}\eta(2^{-j}\lambda)=1,\,\,\lambda>0.
\end{equation}  Fixing $\psi_0(\lambda):=\sum_{j=-\infty}^{0} \eta_j(\lambda),$ and for $j\geq 1,$ $\psi_j(\lambda):=\eta(2^{-j}\lambda),$ we  have
\begin{equation}
    \sum_{\ell=0}^{\infty}\psi_\ell(\lambda)=1,\,\,\lambda>0.
\end{equation} Let $\mathcal{R}$ be a positive Rockland operator and let us  define the family of operators $\psi_j(\mathcal{R})$ using the functional calculus. Then, for $0<q<\infty,$ and $1<p<\infty,$ the Triebel-Lizorkin space $F^{r,\mathcal{R}}_{p,q}(G)$ consists of the distributions $f\in \mathscr{D}'(G)$ such that
\begin{equation*}
 \Vert f\Vert_{F^{r,\mathcal{R}}_{p,q}(G)}:=   \left\Vert\left(\sum_{\ell=0}^{\infty}2^{\frac{\ell r q}{\nu}}\left|\psi_{\ell}(\mathcal{R})f\right|^{q}\right)^{\frac{1}{q}} \right\Vert_{L^p(G)}<\infty,
\end{equation*}and for $p=1,$ the weak-$F^{r}_{1,q}(G)$ space  is defined by the distributions $f\in \mathscr{D}'(G)$ such that
\begin{equation}
 \Vert f \Vert_{\textrm{weak-}F^{r,\mathcal{R}}_{1,q}(G)}    :=   \sup_{t>0}t\left|\left\{x\in G:\left(\sum_{\ell=0}^\infty2^{\frac{\ell rq}{\nu}}|\psi_\ell(\mathcal{R}) f(x)|^q   \right)^{\frac{1}{q}}>t \right\}\right|<\infty.
\end{equation}
\begin{remark}
In the formulation of the Triebel-Lizorkin spaces we use (smooth) dyadic decompositions instead of characteristic functions of intervals because, same as in $\mathbb{R}^n,$ characteristics functions applied to Rockland operators are in general unbounded operators on $L^p(G),$ see \cite{CardonaFeff,Fef} for instance.
\end{remark}
In the following theorem we study some embedding properties for Triebel-Lizorkin spaces and we show, that they are independent on the choice of the positive Rockland operator $\mathcal{R}.$ For a consistent investigation of Triebel-Lizorkin spaces on compact Lie groups, we refer the reader to \cite{NRT} (and to  \cite[Chapter 6]{CardonaRuzhanskySubellipticBesov} for the Triebel-Lizorkin spaces associated to sub-Laplacians).

\begin{theorem}\label{independence}
 Let $G$ be a graded Lie group and let $\mathcal{R}_1$ and $\mathcal{R}_2$ be two positive Rockland operators on $G$ of homogeneous degrees $\nu_1>0$ and $\nu_2>0,$ respectively.  Then we have the following properties.
\begin{itemize}
    \item[(1)] For $1<p,q<\infty,$ $$F^{r}_{p,q}(G):=F^{r,\mathcal{R}_1}_{p,q}(G)=F^{r,\mathcal{R}_2}_{p,q}(G)$$ and for $p=1,$ $$\textrm{weak-}F^{r}_{1,q}(G):=\textrm{weak-}F^{r,\mathcal{R}_1}_{1,q}(G)= \textrm{weak-}F^{r,\mathcal{R}_2}_{1,q}(G),$$ where the coincidence of these spaces is understood in the sense that the topologies induced by their norms are equivalent.
    \item[(2)] $F^{r+\varepsilon,\mathcal{L} }_{p,q_1}(G)\hookrightarrow F^{r,\mathcal{L} }_{p,q_1}(G) \hookrightarrow F^{r,\mathcal{L} }_{p,q_2}(G) \hookrightarrow F^{r,\mathcal{L} }_{p,\infty}(G),$ $\varepsilon >0,$ $0\leq p\leq \infty,$ $0\leq q_1\leq q_2\leq \infty.$
    \item[(3)] $F^{r+\varepsilon,\mathcal{L} }_{p,q_1}(G) \hookrightarrow F^{r,\mathcal{L} }_{p,q_2}(G), $ $\varepsilon >0,$ $0\leq p\leq \infty,$ $1\leq q_2< q_1< \infty.$
    \item[(4)] $F^{r}_{p,2}(G)=L^p_r(G)$ for all $r\in \mathbb{R},$ and all $1<p<\infty,$ where $L^p_r(G)$ are Sobolev spaces on $G.$
\end{itemize}
 \end{theorem}
\begin{proof}
The proof of (2) and  (3) are only an adaptation of the arguments presented in Triebel \cite{Triebel1983}. Because $F^{0}_{p,2}(G)=L^p(G),$ in view of the Littlewood-Paley theorem (Theorem \ref{LPT}) and the fact that $(1+\mathcal{R}_1)^{\frac{r}{\nu_1}}: F^{r}_{p,2}(G)\rightarrow F^{0}_{p,2}(G), $ and  $(1+\mathcal{R}_1)^{\frac{r}{\nu_1}}: L^p_r(G)\rightarrow L^p(G)$ are isomorphisms for any $r\in \mathbb{R}$, we conclude that   $F^{r}_{p,2}(G)=L^p_r(G)$ proving (4). So, to conclude the proof of the theorem, we will prove (1). Let us define the positive Rockland operator $\mathcal{R}:=\mathcal{R}_1^{\nu_2}+\mathcal{R}_2^{\nu_1}.$ Let us define for any $\ell\geq 0,$ the operator $\psi_{\ell}(\mathcal{R})$ by using the functional calculus. So, the relation 
\begin{equation}
    \sum_{\ell=0}^{\infty}\psi_\ell(\lambda)=1,\,\,\lambda>0,
\end{equation}implies that $ \sum_{\ell=0}^{\infty}\psi_\ell(\mathcal{R})=I$ converges to the identity operator $I$ on $L^2(G)$ in the strong topology induced by the norm of $L^2(G).$ Observe that from the properties of the supports of the dyadic decomposition $\psi_{\ell},$ $\ell\in \mathbb{N},$ for $\ell_1,\ell_2$ $\ell_3\in \mathbb{N},$ and the fact that $\textnormal{supp}(\psi_{\ell_i})\subset[2^{\ell_i-1},2^{\ell_i+1}],$ $1\leq i\leq 3,$ we have
 \begin{equation}\label{disjoint}
     \psi_{\ell_1}(\mathcal{R}_1)\psi_{\ell_3}(\mathcal{R})\psi_{\ell_2}(\mathcal{R}_2)\equiv 0,\,\,\,|\ell_2-\ell_1|,\,|\ell_3-\ell_1|\geq M_{0},
 \end{equation}for some  $M_{0}\in \mathbb{N},$ independent of $\ell_{1},\ell_2,\ell_3\in \mathbb{N}_0.$ These can be deduced from the properties of the spectral measures $dE_{i}$ associated to $\mathcal{R}_i$ and $dE$ associated to $\mathcal{R}$ respectively. Indeed, by following \cite[Page 20]{FR}, for all $a,b>0,$
 \begin{equation}
     E(-\infty,a^{\nu_1}]E_1[a,\infty)\equiv 0,\,\,\textnormal{  and  }  E_1(b^{\frac{1}{\nu_2}},\infty)E[b,\infty)\equiv 0.
 \end{equation}Because  $$\psi_{\ell_i}(\mathcal{R}_i)=E_{i}[2^{\ell_i-1},2^{\ell_i+1}]\psi_{\ell_i}(\mathcal{R}_i)E_{i}[2^{\ell_i-1},2^{\ell_i+1}],$$
and 
$$ \,\,\psi_{\ell_3}(\mathcal{R})=E[2^{\ell_3-1},2^{\ell_3+1}]\psi_{\ell_3}(\mathcal{R})E[2^{\ell_3-1},2^{\ell_3+1}],   $$
the existence of $M_0\in \mathbb{N}$ in \eqref{disjoint} follows. Now, let us use this analysis to prove that $F^{0,\mathcal{R}_1}_{p,q}(G)=F^{0,\mathcal{R}_2}_{p,q}(G)$  for $1<p<\infty,$ and that  $\textrm{weak-}F^{0,\mathcal{R}_1}_{1,q}(G)= \textrm{weak-}F^{0,\mathcal{R}_2}_{1,q}(G).$ From this case, the coincidence of the spaces $F^{r,\mathcal{R}_1}_{p,q}(G)=F^{r,\mathcal{R}_2}_{p,q}(G)$ and $\textrm{weak-}F^{r,\mathcal{R}_1}_{1,q}(G)= \textrm{weak-}F^{t,\mathcal{R}_2}_{1,q}(G)$ follows from the fact that $(1+\mathcal{R}_i)^{\frac{\nu}{r_i}}:F^{r,\mathcal{R}_i}_{p,q}(G)\rightarrow F^{0,\mathcal{R}_i}_{p,q}(G)$ is an isomorphism, for any $r\in \mathbb{R}.$ Now in view of \eqref{disjoint},  let us observe that for $1<q<\infty,$ we can estimate
\begin{align*}
   \left(\sum_{\ell_1=0}^{\infty}\left|\psi_{\ell_1}(\mathcal{R}_1)f\right|^{q}\right)^{\frac{1}{q}} &\leq \left(\sum_{\ell_1=0}^{\infty}\left|\sum_{\ell_2,\ell_3=0}^{\infty}\psi_{\ell_1}(\mathcal{R}_1)\psi_{\ell_3}(\mathcal{R})\psi_{\ell_2}(\mathcal{R}_2)f\right|^{q}\right)^{\frac{1}{q}}\\
   &\leq   \left(\sum_{\ell_1=0}^{\infty}\sum_{\ell_2,\ell_3=0}^{\infty}\left|\psi_{\ell_1}(\mathcal{R}_1)\psi_{\ell_3}(\mathcal{R})\psi_{\ell_2}(\mathcal{R}_2)f\right|^{q}\right)^{\frac{1}{q}}\\
   &\leq   \left(\sum_{\ell_1=0}^{\infty}\sum_{|\ell_2-\ell_1|,\,|\ell_3-\ell_1|\leq M_0}\left|\psi_{\ell_1}(\mathcal{R}_1)\psi_{\ell_3}(\mathcal{R})\psi_{\ell_2}(\mathcal{R}_2)f\right|^{q}\right)^{\frac{1}{q}}.
\end{align*}So, we have
\begin{equation}\label{Differences}
   \left(\sum_{\ell_1=0}^{\infty}\left|\psi_{\ell_1}(\mathcal{R}_1)f\right|^{q}\right)^{\frac{1}{q}} \leq   \left(\sum_{\ell_1=0}^{\infty}\sum_{|\ell_2-\ell_1|,\,|\ell_3-\ell_1|\leq M_0}\left|\psi_{\ell_1}(\mathcal{R}_1)\psi_{\ell_3}(\mathcal{R})\psi_{\ell_2}(\mathcal{R}_2)f\right|^{q}\right)^{\frac{1}{q}}.  
\end{equation}
Now, taking the $L^p$-norm (or the $L^{1,\infty}$-norm) in both sides of \eqref{Differences}, and by applying the vector-valued inequality \eqref{1pinfty}, two times, we obtain
\begin{align*}
     \Vert f\Vert_{F^{0,\mathcal{R}_1}_{p,q}}&=\left\Vert  \left(\sum_{\ell_1=0}^{\infty}\left|\psi_{\ell_1}(\mathcal{R}_1)f\right|^{q}\right)^{\frac{1}{q}}   \right\Vert_{L^{p}(G)}\\
     &\leq  \left\Vert  \left(\sum_{\ell_1=0}^{\infty}\sum_{|\ell_2-\ell_1|,\,|\ell_3-\ell_1|\leq M_0}\left|\psi_{\ell_1}(\mathcal{R}_1)\psi_{\ell_3}(\mathcal{R})\psi_{\ell_2}(\mathcal{R}_2)f\right|^{q}\right)^{\frac{1}{q}}   \right\Vert_{L^{p}(G)}\\
     &\lesssim_{p,\, M_0}\left\Vert  \left(\sum_{\ell_1=0}^{\infty}\sum_{\,|\ell_2-\ell_1|\leq M_0}\left|\psi_{\ell_1}(\mathcal{R})\psi_{\ell_2}(\mathcal{R}_2)f\right|^{q}\right)^{\frac{1}{q}}   \right\Vert_{L^{p}(G)}\\
     &\lesssim_{p,\, M_0}\left\Vert  \left(\sum_{\ell_1=0}^{\infty}\left|\psi_{\ell_1}(\mathcal{R}_2)f\right|^{q}\right)^{\frac{1}{q}}   \right\Vert_{L^{p}(G)}\\
     &=\Vert f\Vert_{F^{0,\mathcal{R}_2}_{p,q}}.
\end{align*}This proves that $\Vert f\Vert_{F^{0,\mathcal{R}_1}_{p,q}} \lesssim_{p,\, M_0}\Vert f\Vert_{F^{0,\mathcal{R}_2}_{p,q}}.$ In a similar way we can prove that  $\Vert f\Vert_{F^{0,\mathcal{R}_2}_{p,q}} \lesssim_{p,\, M_0}\Vert f\Vert_{F^{0,\mathcal{R}_1}_{p,q}}.$ The same analysis applying the $L^{1,\infty}$-norm in both sides of  \eqref{Differences}, and using the weak vector-valued inequality \eqref{weakvectorvaluedinequality} imply that  $$\Vert f\Vert_{\textrm{weak-}F^{0,\mathcal{R}_1}_{1,q}(G)} \lesssim_{p,\, M_0}\Vert f\Vert_{\textrm{weak-}F^{0,\mathcal{R}_2}_{1,q}(G)} \lesssim_{p,\, M_0} \Vert f\Vert_{\textrm{weak-}F^{0,\mathcal{R}_1}_{1,q}(G)}.$$ So, we have proved (1). The proof of Theorem \ref{independence} is complete.
\end{proof}

\section{H\"ormander-Mihlin multipliers on Triebel-Lizorkin spaces: Proof of Theorem \ref{HMTTL}}

Let us fix $\eta\in C^{\infty}_0(\mathbb{R}^{+},[0,1]),$ $\eta\neq 0,$ so that $\textnormal{supp}(\eta)\subset [1/2,2],$ and such that 
\begin{equation}
    \sum_{j\in \mathbb{Z}}\eta(2^{-j}\lambda)=1,\,\,\lambda>0.
\end{equation}
By defining $\psi_0(\lambda):=\sum_{j=-\infty}^{0} \eta_j(\lambda),$ and for $j\geq 1,$ $\psi_j(\lambda):=\eta(2^{-j}\lambda),$ we obviously have
\begin{equation}
    \sum_{\ell=0}^{\infty}\psi_\ell(\lambda)=1,\,\,\lambda>0.
\end{equation}

\begin{proof}[Proof of Theorem \ref{HMTTL}] 
By observing that  $(1+\mathcal{R})^{\frac{r}{\nu}}:F^{r}_{p,q}(G)\rightarrow F^{0}_{p,q}(G)$ and $ (1+\mathcal{R})^{-\frac{r}{\nu}}: F^{0}_{p,q}(G)\rightarrow F^{r}_{p,q}(G)$ are isomorphism,  it is suffices to prove that $A$ admits a bounded extension from $ F^{0}_{p,q}(G)$ into $ F^{0}_{p,q}(G).$ For this, let us define the vector-valued operator $W:L^2(G,\mathcal{\ell}^2(\mathbb{N}_0))\rightarrow L^2(G,\mathcal{\ell}^2(\mathbb{N}_0))$ by
\begin{equation}
    W(\{g_{\ell}\}_{\ell=0}^{\infty}):=  (\{W_\ell g_{\ell}\}_{\ell=0}^{\infty}),\,\,W_{\ell}:=A\psi_{\ell}(\mathcal{R}).
\end{equation} Observe that $W$ is well-defined (bounded from $L^2(G,\mathcal{\ell}^2(\mathbb{N}_0)$ into  $L^2(G,\mathcal{\ell}^2(\mathbb{N}_0)$)), because $A$ admits a bounded extension on $L^2(G)$ and also, in view of the following estimate
\begin{align*}
  \Vert W(\{g_{\ell}\}_{\ell=0}^{\infty})\Vert^2_{ L^2(G,\mathcal{\ell}^2(\mathbb{N}_0))}
  &:=  \int\limits_{G} \sum_{\ell=0}^\infty |A\psi_\ell(\mathcal{R}) g_\ell(x)|^2dx\\
  &=\sum_{\ell=0}^\infty\int\limits_{G}|A\psi_\ell(\mathcal{R}) g_\ell(x)|^2dx\\
    &\leq \Vert A \Vert^2_{\mathcal{L}(L^2(G))}\sup_{\ell}\Vert \psi_\ell(\mathcal{R}) \Vert^2_{\mathcal{L}(L^2(G))}\sum_{\ell=0}^\infty\int\limits_{G}  | g_\ell(x)|^2dx\\
    &\leq \Vert A \Vert^2_{\mathcal{L}(L^2(G))}\sup_{\ell}\Vert \psi_\ell \Vert^2_{L^{\infty}(\mathbb{R})}\sum_{\ell=0}^\infty\int\limits_{G}  | g_\ell(x)|^2dx\\
    &= \Vert A \Vert^2_{\mathcal{L}(L^2(G))}\Vert \psi \Vert^2_{L^{\infty}(\mathbb{R})}\sum_{\ell=0}^\infty\int\limits_{G}  | g_\ell(x)|^2dx\\
    &\lesssim  \Vert \{g_{\ell}\}_{\ell=0}^{\infty}\Vert^2_{ L^2(G,\mathcal{\ell}^2(\mathbb{N}_0))}. 
\end{align*}
So, observe that, in order to prove Theorem \ref{HMTTL} it is enough to prove the following two lemmas.
\begin{lemma}\label{Lemma1}
$W:L^q(G,\mathcal{\ell}^q(\mathbb{N}_0))\rightarrow L^q(G,\mathcal{\ell}^q(\mathbb{N}_0))$ admits a bounded extension for all $1<q<\infty.$
\end{lemma}
\begin{lemma}\label{Lemma2}
$W:L^1(G,\mathcal{\ell}^q(\mathbb{N}_0))\rightarrow L^{1,\infty}(G,\mathcal{\ell}^q(\mathbb{N}_0))$ admits a bounded extension for all $1<q<\infty.$
\end{lemma}

\begin{remark}\label{RemarkInterpolation}Indeed, by the Marcinkiewicz interpolation, these two lemmas are enough to show that $W:L^p(G,\mathcal{\ell}^q(\mathbb{N}_0))\rightarrow L^{p}(G,\mathcal{\ell}^q(\mathbb{N}_0))$ admits a bounded extension for all $1<p\leq q<\infty.$ The case $1<q\leq p<\infty$ follows from the fact that  $L^{p'}(G,\mathcal{\ell}^{q'}(\mathbb{N}_0))$ is the dual of $L^{p}(G,\mathcal{\ell}^{q}(\mathbb{N}_0))$ and also that  Lemma \ref{Lemma1} and  Lemma \ref{Lemma2} hold if we change $A$ by its standard $L^2$-adjoint. Now, note that by defining $\psi_{-1}=\psi_0,$ we have
\begin{align*}
  \Vert Af\Vert_{F^{0}_{p,q}(G)}&=  \left\Vert \left(\sum_{l=0}^{\infty} |A\psi_{l}(\mathcal{R})f|^{q}    \right)^{\frac{1}{q}}\right\Vert_{L^{p}(G)}\\
  &\leq    \left\Vert \left(\sum_{l=0}^{\infty} |A\psi_{l}(\mathcal{R})[\psi_{l-1}(\mathcal{R})+\psi_{l}(\mathcal{R})+\psi_{l+1}(\mathcal{R})] f|^{q}    \right)^{\frac{1}{q}}\right\Vert_{L^{p}(G)}\\
  &=\Vert W(\{[\psi_{l-1}(\mathcal{R})+\psi_{l}(\mathcal{R})+\psi_{l+1}(\mathcal{R})] f\}_{l=0}^{\infty}) \Vert_{L^p(\ell^q)}\\
  &\lesssim \Vert \{ [\psi_{l-1}(\mathcal{R})+\psi_{l}(\mathcal{R})+\psi_{l+1}(\mathcal{R})] f\}_{l=0}^{\infty}) \Vert_{L^p(\ell^q)} \\
  &\lesssim \Vert f \Vert_{F^{0}_{p,q}(G)}. 
\end{align*} Also note that from Lemma \ref{Lemma2}, we have

\begin{align*}
  \Vert Af\Vert_{\textrm{weak-}F^{0}_{1,q}(G)}&=  \left\Vert \left(\sum_{l=0}^{\infty} |A\psi_{l}(\mathcal{R})f|^{q}    \right)^{\frac{1}{q}}\right\Vert_{L^{1,\infty}(G)}\\
  &\leq    \left\Vert \left(\sum_{l=0}^{\infty} |A\psi_{l}(\mathcal{R})[\psi_{l-1}(\mathcal{R})+\psi_{l}(\mathcal{R})+\psi_{l+1}(\mathcal{R})] f|^{q}    \right)^{\frac{1}{q}}\right\Vert_{L^{1,\infty}(G)}\\
  &=\Vert W(\{[\psi_{l-1}(\mathcal{R})+\psi_{l}(\mathcal{R})+\psi_{l+1}(\mathcal{R})] f\}_{l=0}^{\infty}) \Vert_{L^{1,\infty}(\ell^q)}\\
  &\lesssim \Vert \{ [\psi_{l-1}(\mathcal{R})+\psi_{l}(\mathcal{R})+\psi_{l+1}(\mathcal{R})] f\}_{l=0}^{\infty}) \Vert_{L^1(\ell^q)} \\
  &\lesssim \Vert f \Vert_{F^{0}_{1,q}(G)}. 
\end{align*}
So, knowing that Lemmas \ref{Lemma1} and \ref{Lemma2} are enough for proving Theorem \ref{HMTTL}, we will proceed with their proofs. 
\end{remark}
\begin{proof}[Proof of Lemma \ref{Lemma1}] It suffices to prove that the operators $W_{\ell}$ are uniformly bounded on $L^q(G).$ This is trivial for $q=2,$ so it is suffices (by the duality argument) that the operators $W_{\ell}$ are uniformly bounded from $L^1(G)$ into $L^{1,\infty}(G).$ So, we are going to prove that there exists a constant $C>0,$ independent of $f\in L^1(G),$ and  $\ell\in \mathbb{N}_0,$ such that
\begin{equation}\label{weak(1,1)inequalityWell}
    \left|\left\{x\in G:|W_\ell f(x)|  >t \right\}\right|\leq \frac{C}{t}\Vert f\Vert_{L^1(G)}.
\end{equation}

We start the proof by applying the Calder\'on-Zygmund decomposition Lemma to the non-negative function $f\in L^p(G)\cap L^1(G)\subset L^1(G),$ under the identification $G\simeq \mathbb{R}^n,$ (see, e.g. Hebish \cite{Hebish}) in order to obtain a suitable family of  disjoint open sets  $\{I_j\}_{j=0}^{\infty}$  such that
\begin{itemize}
    \item $f(x)\leq t,$ for $a.e.$ $x\in G\setminus \cup_{j\geq 0}I_j,$\\
    
    \item $\sum_{j\geq 0}|I_j|\leq \frac{C}{t}\Vert f\Vert_{L^1(G)},$ and\\ 
    
    \item $t|I_j|\leq \int_{I_j}f(x)dx\leq 2|I_j|t,$ for all $j.$ 
\end{itemize}Now, for every $j\in \mathbb{N}_0,$ let us define $R_j$ by
\begin{equation}\label{Rj}
    R_{j}:=\sup\{R>0: B(z_j,R)\subset I_j, \textnormal{   for some   }z_j\in I_j\},
\end{equation} where $B(z_j,R)=\{x\in I_j:|z_j^{-1}x|<R\}.$ Then, we can assume that every $I_j$ is bounded, and that  $I_j\subset B(z_j,2R_j),$ where $z_j\in I_j$ (see Hebish \cite{Hebish}). 
\begin{remark}
Before  continuing with the proof note that by assuming $f(e_G)>t,$ (this is just re-defining $f\in L^p(G)\cap L^1(G)$ at the identity element) we should have that
\begin{equation}\label{eGasump}
    e_{G}\in \bigcup_{j}I_j,
\end{equation} because $f(x)\leq t,$ for $a.e.$ $x\in G\setminus \cup_{j\geq 0}I_j.$
\end{remark}
Let us define, for every $x\in I_j,$
\begin{equation}
    g(x):=\frac{1}{|I_j|}\int\limits_{I_j}f(y)dy,\,\,\,b(x)=f(x)-g(x),
\end{equation} and for $x\in G\setminus\cup_{j\geq 0}I_j,$
\begin{equation}
    g(x)=f(x),\,\,\, b(x)=0.
\end{equation} Observe that for every $x\in I_j,$ 
\begin{align*}
    |g(x)|= \left|   \frac{1}{|I_j|}\int\limits_{I_j}f(y)dy \right|\leq 2t.
\end{align*}
By applying the Minkowski inequality, we have \begin{align*}
   & \left|\left\{x\in G:|W_{\ell} f(x)|>t \right\}\right|\leq \left|\left\{x\in G:|W_{\ell} g(x)|>\frac{t}{2} \right\} \right|\\
    &\hspace{3cm}+\left|\left\{x\in G:|W_{\ell} b(x)|>\frac{t}{2} \right\}\right|.
\end{align*}
By the Chebyshev inequality, we have
\begin{align*}
    & \left|\left\{x\in G:|W_\ell f(x)|>t \right\}\right|\\
    & \leq \left|\left\{x\in G:|W_{\ell} g(x)|>\frac{t}{2} \right\} \right|+\left|\left\{x\in G:|W_{\ell} b(x)|>\frac{t}{2} \right\}\right|\\
    &=  \left|\left\{x\in G:|W_{\ell} g(x)|^2>\frac{t^{2}}{2^2} \right\} \right|+\left|\left\{x\in G:|W_{\ell} b(x)|>\frac{t}{2} \right\}\right|\\
    &\leq \frac{2^2}{t^2}\int\limits_{G}|W_\ell g(x)|^2dx+ \left|\left\{x\in G:|W_{\ell}b(x)|>\frac{t}{2} \right\}\right|\\
    &\leq \frac{2^2}{t^2}\sup_{\ell}\Vert W_{\ell}\Vert_{\mathcal{L}(L^2(G))}\int\limits_{G}| g(x)|^2dx+ \left|\left\{x\in G:|W_{\ell}b(x)|>\frac{t}{2} \right\}\right|\\
    &\lesssim \frac{2^2}{t^2}\int\limits_{G}| g(x)|^2dx+ \left|\left\{x\in G:|W_{\ell}b(x)|>\frac{t}{2} \right\}\right|,
\end{align*}
in view of the $L^2(G)$-boundedness of $A$ and the fact that the operators $\psi_{\ell}(\mathcal{R})$ are $L^2(G)$-bounded uniformly  in $\ell.$ Also, note that the  estimate
\begin{align*}
    \Vert g\Vert_{L^2(G)}^2&=\int\limits_{G}|g(x)|^2dx=\sum_{j}\int\limits_{I_j}|g(x)|^2dx+\int\limits_{G\setminus \cup_{j}I_j}|g(x)|^2dx\\
    &=\sum_{j}\int\limits_{I_j}|g(x)|^2dx+\int\limits_{G\setminus \cup_{j}I_j}|f(x)|^2dx\\
    &\leq \sum_{j}\int\limits_{I_j}(2t)^{2}dx+\int\limits_{G\setminus \cup_{j}I_j}f(x)^2dx\lesssim t^{2}\sum_{j}|I_j|+\int\limits_{G\setminus \cup_{j}I_j}f(x)f(x)dx\\
    &\leq t^{2}\times \frac{C}{t}\Vert f\Vert_{L^1(G)}+t\int\limits_{G\setminus \cup_{j}I_j}f(x)dx\lesssim t\Vert f\Vert_{L^1(G)},
\end{align*} implies that,
\begin{align*}
  \Small{ \left|\left\{x\in G:|W_\ell f(x)|>t \right\}\right|\leq \frac{4}{t}\Vert f\Vert_{L^1(G)}+ \left|\left\{x\in G:|W_\ell b(x)|^2   >\frac{t}{2} \right\}\right|.}
\end{align*}
Taking into account that $b\equiv 0$ on $G\setminus \cup_j I_j,$ we have that
\begin{equation}
    b=\sum_{k}b_k,\,\,\,b_k(x)=b(x)\cdot 1_{I_k}(x).
\end{equation} Let us assume that $I_{j}^*$ is an open set, such that $I_j\subset I_j^*,$ and $|I_{j}^*|=K|I_{j}|$ for some $K>0,$ and $\textnormal{dist}(\partial I_{j}^*,\partial I_{j})\geq 4c\,\textnormal{dist}(\partial I_{j},e_{G}),$ where $c>0$ and $e_{G}$ is the identity element of $G$.  So, by the Minkowski inequality we have,
\begin{align*}
     & \left|\left\{x\in G:|W_\ell b(x)|  >\frac{t}{2} \right\}\right|\\
      &=\left|\left\{x\in \cup_j I_j^*:|W_\ell b(x)|   >\frac{t}{2} \right\}\right|+\left|\left\{x\in G\setminus  \cup_j I_j^*:|W_\ell b(x)|   >\frac{t}{2} \right\}\right|\\
      &\leq \left|\left\{x\in G:x\in \cup_j I_j^* \right\}\right|+\left|\left\{x\in G\setminus  \cup_j I_j^*:|W_\ell b(x)|   >\frac{t}{2} \right\}\right|.
      \end{align*} In consequence, we have the estimates,
      \begin{align*} & \left|\left\{x\in G:|W_\ell b(x)|  >\frac{t}{2} \right\}\right|\leq\sum_{j}|I_j^*|+\left|\left\{x\in G\setminus  \cup_j I_j^*:|W_\ell b(x)|   >\frac{t}{2} \right\}\right|\\
      &=K\sum_{j}|I_j|+\left|\left\{x\in G\setminus  \cup_j I_j^*:|W_\ell b(x)|   >\frac{t}{2} \right\}\right|\\
      &\leq \frac{CK}{t}\Vert f\Vert_{L^1(G)}+\left|\left\{x\in G\setminus  \cup_j I_j^*:|W_\ell b(x)|   >\frac{t}{2} \right\}\right|.
  \end{align*}The Chebyshev inequality allows to estimate the right hand side above as follows,
  \begin{align*}
     & \left|\left\{x\in G\setminus  \cup_j I_j^*:|W_\ell b(x)|   >\frac{t}{2} \right\}\right|\leq\frac{2}{t}\int\limits_{ G\setminus  \cup_j I_j^*} |W_{\ell}b(x)|dx\\
     &\leq  \frac{2}{t}\sum_{k}\int\limits_{ G\setminus  \cup_j I_j^*} |W_{\ell}b_k(x)|dx.
  \end{align*}
From now, let us denote by $\kappa_\ell$  the right convolution kernel of     $W_\ell:=A\psi_\ell(\mathcal{R}) .$ Observe that
\begin{align*}
     &\left|\left\{x\in G\setminus  \cup_j I_j^*:|W_\ell b(x)|   >\frac{t}{2} \right\}\right|\leq \frac{2}{t}\sum_{k}\int\limits_{ G\setminus  \cup_j I_j^*} |W_{\ell}b_k(x)|dx\\
     &= \frac{2}{t}\sum_{k}\int\limits_{ G\setminus  \cup_j I_j^*}\left|b_k   \ast \kappa_{\ell}(x)\right|dx\\
     &= \frac{2}{t}\sum_{k}\int\limits_{ G\setminus  \cup_j I_j^*}\left|\int\limits_{I_k}b_k(z)\kappa_{\ell}(z^{-1}x)dz   \right|dx.
     \end{align*} By using that the average of $b_{k}$ on $I_k$ is zero, $\int_{I_k}b_{k}(z)dz=0,$ we have
     \begin{align*}
     \frac{2}{t}\sum_{k} \int\limits_{ G\setminus  \cup_j I_j^*}  &\left|\int\limits_{I_k}b_k(z)\kappa_{\ell}(z^{-1}x)dz   \right|dx\\
      &=\frac{2}{t}\sum_{k}\int\limits_{ G\setminus  \cup_j I_j^*}\left|\int\limits_{I_k}b_k(z)\kappa_{\ell}(z^{-1}x)dz  - \kappa_{\ell}(x)\int\limits_{I_k}b_{k}(z)dz \right|dx\\
       &=\frac{2}{t}\sum_{k}\int\limits_{ G\setminus  \cup_j I_j^*}\left|\int\limits_{I_k}(\kappa_{\ell}(z^{-1}x)  - \kappa_{\ell}(x))b_{k}(z) dz\right|dx.
     \end{align*}
 If we assume for a moment that
 \begin{equation}\label{calderonkernels}
    M=\sup_{k} \sup_{z\in I_k}\sum_{\ell=0}^\infty\int\limits_{ G\setminus  \cup_j I_j^*}|\kappa_{\ell}(z^{-1}x)  - \kappa_{\ell}(x)|dx<\infty,
 \end{equation}then we have
     \begin{align*}
          & \left|\left\{x\in G\setminus  \cup_j I_j^*:|W_\ell b(x)|   >\frac{t}{2} \right\}\right|\leq \frac{2M}{t}\sum_{k}\int\limits_{I_k}|b_{k}(z)|dz\\
          &=\frac{2M}{t}\| b\|_{L^1(G)}\leq \frac{6M}{t}\| f\|_{L^1(G)}.
     \end{align*}
     So, if we prove the estimate \eqref{calderonkernels} we obtain the weak (1,1) inequality \eqref{weak(1,1)inequality} for $f\in L^p(G)\cap L^1(G),$ $f\geq 0$. For the proof of \eqref{calderonkernels} let us use the estimates of the Calder\'on-Zygmund kernel of every operator $W_{\ell}.$ Let us point out that (in view of \eqref{eGasump} and from \cite[Page 17]{CardonaRuzhanskyBesovSpaces}) for $x\in G\setminus \cup_jI_j^*,$   and  $z\in I_{k},$ $4c|z|=4c\times \textnormal{dist}(z,e_G)\lesssim  \textnormal{dist}(\partial I_{k}^*,\partial I_{k})  \leq|x|.$ 
     So, by a suitable variable change of variables  and by using \eqref{ThLiWeak}, we have 
     \begin{align*}
        M_k:= \sup_{z\in I_k}\sum_{\ell=0}^\infty\int\limits_{ G\setminus  \cup_j I_j^*} & |\kappa_{\ell}(z^{-1}x)  - \kappa_{\ell}(x)|dx\\
         &=\sup_{z\in I_k}\sum_{\ell=0}^\infty\int\limits_{ G\setminus  \cup_j I_j^*}|2^{-\ell Q}\kappa_{\ell}(2^{-\ell}\cdot z^{-1}x)  - 2^{-\ell Q}\kappa_{\ell}(2^{-\ell }\cdot x)|dx\\
         &\leq \sup_{z\in I_k} \sum_{\ell=0}^\infty\int\limits_{ |x|>4c|z|}|2^{-\ell Q}\kappa_{\ell}(2^{-\ell}\cdot z^{-1}x)  - 2^{-\ell Q}\kappa_{\ell}(2^{-\ell }\cdot x)|dx\\
         &\leq \sum_{\ell=0}^\infty \sup_{z\in G}           \int\limits_{ |x|>4c|z|}|2^{-\ell Q}\kappa_{\ell}(2^{-\ell}\cdot z^{-1}x)  - 2^{-\ell Q}\kappa_{\ell}(2^{-\ell }\cdot x)|dx\\
          &=  \sum_{\ell=0}^\infty \mathscr{I}_{\ell}\lesssim \sum_{\ell=0}^\infty  2^{-\ell\varepsilon_0}=O(1).
     \end{align*}Because $$ M_k:= \sup_{z\in I_k}\sum_{\ell=0}^\infty\int\limits_{ G\setminus  \cup_j I_j^*}  |\kappa_{\ell}(z^{-1}x)  - \kappa_{\ell}(x)|dx\lesssim \sum_{\ell=0}^\infty  2^{-\ell\varepsilon_0},$$ with the right hand side of the inequality being independent of $k,$ we conclude that $M$ in \eqref{calderonkernels} is finite. So, we have prove \eqref{weak(1,1)inequalityWell} for $f\in L^p(G)\cap L^1(G)$ with $f\geq 0.$
     Note that if $f\in L^p(G)\cap L^1(G)$ is real-valued, one can decompose $f=f^+-f^-,$ as the difference of two non-negative functions, where $f^+,f^-\in L^p(G)\cap L^1(G), $ and $|f|=f^++f^{-}.$ Because $f^+,f^-\leq |f|,$ we have
     \begin{align*}
    \left|\left\{x\in G:|W_\ell f(x)|  >t \right\}\right|\\
    &\leq  \left|\left\{x\in G:|W_\ell f_+(x)|  >\frac{t}{2} \right\}\right| +\left|\left\{x\in G:|W_\ell f_{-}(x)|  >\frac{t}{2} \right\}\right|\\
    &\leq  \frac{C}{t}\Vert f_+\Vert_{L^1(G)}+\frac{C}{t}\Vert f_{-}\Vert_{L^1(G)}\\
    &\leq  \frac{2C}{t}\Vert f\Vert_{L^1(G)}.
\end{align*}A similar analysis, by splitting a complex function $f\in  L^p(G)\cap L^1(G)$ into its real and imaginary parts allows to conclude the weak (1,1) inequality \eqref{weak(1,1)inequalityWell} to complex functions.
     Thus, the proof of Lemma \ref{Lemma1} is complete. 
\end{proof}     
\begin{proof}[Proof of Lemma \ref{Lemma2}]
Now, we claim that 
\begin{equation}\label{weakvectorvalued}
    W:L^{1}(G,\ell^r(\mathbb{N}_0))\rightarrow L^{1,\infty}(G,\ell^r(\mathbb{N}_0)),\,\,\,1<r<\infty.
\end{equation}
extends to a bounded operator. For the proof of \eqref{weakvectorvalued}, we need to show that there exists a constant $C>0$ independent of $\{f_\ell\}\in L^1(G,\ell^r(\mathbb{N}_0))$ and $t>0,$ such that 
\begin{equation}
    \left|\left\{x\in G:\left(\sum_{\ell=0}^\infty|W_\ell f_\ell(x)|^r   \right)^{\frac{1}{r}}>t \right\}\right|\leq \frac{C}{t}\Vert \{f_\ell\} \Vert_{L^1(G,\ell^r(\mathbb{N}_0)}.
\end{equation} So, fix $\{f_\ell\}\in L^1(G,\ell^r(\mathbb{N}_0))$ and $t>0,$ and let $h(x):= \left(\sum_{\ell=0}^\infty| f_\ell(x)|^r   \right)^{\frac{1}{r}}, $ apply the Calder\'on-Zygmund decomposition Lemma to $h\in L^1(G),$ under the identification $G\simeq \mathbb{R}^n,$  (see e.g. Hebish \cite{Hebish}) in order to obtain a disjoint collection $\{I_j\}_{j=0}^{\infty}$ of disjoint open sets such that
\begin{itemize}
    \item $h(x)\leq t,$ for $a.e.$ $x\in G\setminus \cup_{j\geq 0}I_j,$\\
    
    \item $\sum_{j\geq 0}|I_j|\leq \frac{C}{t}\Vert h\Vert_{L^1(G)},$ and\\ 
    
    \item $t\leq \frac{1}{|I_j|}\int_{I_j}h(x)dx\leq 2t,$ for all $j.$
\end{itemize} Now, we will define a suitable decomposition of $f_\ell,$ for every $\ell\geq 0.$ Recall that   every $I_j$ is diffeomorphic to an open cube on $\mathbb{R}^n,$ that it is bounded, and that  $I_j\subset B(z_j,2R_j),$ where $z_j\in I_j$ (see Hebish \cite{Hebish}).  Let us define, for every $\ell,$ and $x\in I_j,$
\begin{equation}
    g_\ell(x):=\frac{1}{|I_j|}\int\limits_{I_j}f_\ell(y)dy,\,\,\,b_{\ell}(x)=f_\ell(x)-g_{\ell}(x).
\end{equation} and for $x\in G\setminus\cup_{j\geq 0}I_j,$
\begin{equation}
    g_\ell(x)=f_\ell(x),\,\,\, b_\ell(x)=0.
\end{equation} So, for a.e.  $x\in G,$ $f_\ell(x)=g_\ell(x)+b_\ell(x).$
Note that for every $1<r<\infty,$ $\Vert \{g_\ell\}\Vert_{L^r(\ell^r)}^r\leq t^{r-1}\Vert \{f_\ell\} \Vert_{L^1(\ell^r)}. $ Indeed, for $x\in I_j$, Minkowski integral inequality gives,
\begin{align*}
    \left(\sum_{\ell=0}^\infty|g_\ell(x)|^r\right)^{\frac{1}{r}} &\leq  \left(\sum_{\ell=0}^\infty\left|   \frac{1}{|I_j|}\int\limits_{I_j}f_\ell(y)dy \right|^r\right)^{\frac{1}{r}}\leq \frac{1}{|I_j|}\int\limits_{I_j}\left(\sum_{\ell=0}^\infty|   f_\ell(y) |^r\right)^{\frac{1}{r}}dy\\
    &=\frac{1}{|I_j|}\int\limits_{I_j}h(y)dy\\
    &\leq 2t.
\end{align*}
Consequently,
\begin{align*}
    \sum_{\ell=0}^\infty|g_\ell(x)|^r\leq (2t)^{r},
\end{align*}
and from the fact that $h(x)\leq t,$ for $a.e.$ $x\in G\setminus \cup_{j\geq 0}I_j,$ we have
\begin{align*}
    \Vert \{g_\ell\}\Vert_{L^r(\ell^r)}^r&=\int\limits_{G}\sum_{\ell=0}^\infty|g_\ell(x)|^rdx=\sum_{j}\int\limits_{I_j}\sum_{\ell=0}^\infty|g_\ell(x)|^rdx+\int\limits_{G\setminus \cup_{j}I_j}\sum_{\ell=0}^\infty|g_\ell(x)|^rdx\\
    &=\sum_{j}\int\limits_{I_j}\sum_{\ell=0}^\infty|g_\ell(x)|^rdx+\int\limits_{G\setminus \cup_{j}I_j}\sum_{\ell=0}^\infty|f_\ell(x)|^rdx\\
    &\leq \sum_{j}\int\limits_{I_j}(2t)^{r}dx+\int\limits_{G\setminus \cup_{j}I_j}h(x)^rdx\\
    &\lesssim t^{r}\sum_{j}|I_j|+\int\limits_{G\setminus \cup_{j}I_j}h(x)^{r-1}h(x)dx\\
    &\leq t^{r}\times \frac{C}{t}\Vert h\Vert_{L^1(G)}+t^{r-1}\int\limits_{G\setminus \cup_{j}I_j}h(x)dx\lesssim t^{r-1}\Vert h\Vert_{L^1(G)}\\
    &=t^{r-1}\Vert \{f_\ell\} \Vert_{L^1(\ell^r)}.
\end{align*}
Now, by using   the Minkowski and the Chebyshev inequality, we obtain 
\begin{align*}
    & \left|\left\{x\in G:\left(\sum_{\ell=0}^\infty|W_{\ell} f_\ell(x)|^r   \right)^{\frac{1}{r}}>t \right\}\right|\\
    &\leq  \left|\left\{x\in G:\left(\sum_{\ell=0}^\infty|W_{\ell} g_\ell(x)|^r   \right)^{\frac{1}{r}}>\frac{t}{2} \right\}\right|+ \left|\left\{x\in G:\left(\sum_{\ell=0}^\infty|W_{\ell} b_\ell(x)|^r   \right)^{\frac{1}{r}}>\frac{t}{2} \right\}\right|\\
    &\leq \frac{2^r}{t^r}\int\limits_{G}\sum_{\ell=0}^\infty|W_{\ell} g_\ell(x)|^rdx+ \left|\left\{x\in G:\left(\sum_{\ell=0}^\infty|W_{\ell} b_\ell(x)|^r   \right)^{\frac{1}{r}}>\frac{t}{2} \right\}\right|.
\end{align*} In view of Lemma \ref{Lemma1}, $W:L^r(G,\ell^r(\mathbb{N}_0))\rightarrow L^r(G,\ell^r(\mathbb{N}_0)),$ extends to a bounded operator and 
\begin{equation}
    \int\limits_{G}\sum_{\ell=0}^\infty|W_{\ell} g_\ell(x)|^rdx=\Vert W\{g_\ell\}\Vert_{L^r(\ell^r)}^r\lesssim  \Vert \{g_\ell\}\Vert_{L^r(\ell^r)}^r\leq t^{r-1}\Vert \{f_\ell\}\Vert_{L^1(\ell^r)}.
\end{equation}
Consequently,
 \begin{align*}
    & \left|\left\{x\in G:\left(\sum_{\ell=0}^\infty|W_{\ell} f_\ell(x)|^r   \right)^{\frac{1}{r}}>t \right\}\right|\\
    &\lesssim  \frac{1}{t}\Vert \{f_\ell\}\Vert_{L^1(\ell^r)}+ \left|\left\{x\in G:\left(\sum_{\ell=0}^\infty|W_{\ell} b_\ell(x)|^r   \right)^{\frac{1}{r}}>\frac{t}{2} \right\}\right|.
\end{align*}
Now, we only need to prove that
\begin{equation}
     \left|\left\{x\in G:\left(\sum_{\ell=0}^\infty|W_{\ell} b_\ell(x)|^r   \right)^{\frac{1}{r}}>\frac{t}{2} \right\}\right|\lesssim  \frac{1}{t}\Vert \{f_\ell\}\Vert_{L^1(\ell^r)}.   
\end{equation} 

Taking into account that $b_\ell\equiv 0$ on $G\setminus \cup_j I_j,$ we have that
\begin{equation}
    b_\ell=\sum_{k}b_{\ell,k},\,\,\,b_{\ell,k}(x)=b_{\ell}(x)\cdot 1_{I_k}(x).
\end{equation} Let us assume that $I_{j}^*$ is a open set, such that $|I_{j}^*|=K|I_{j}|$ for some $K>0,$ and $\textnormal{dist}(\partial I_{j}^*,\partial I_{j})\geq 4c\textnormal{dist}(\partial I_{j},e_{G}),$ where $c$ is defined in \eqref{ThLiWeak} and $e_{G}$ is the identity element of $G$.  So, by the Minkowski inequality we have,
\begin{align*}
     & \left|\left\{x\in G:\left(\sum_{\ell=0}^\infty|W_{\ell} b_{\ell}(x)|^r   \right)^{\frac{1}{r}}>\frac{t}{2} \right\}\right|\\
      &=\left|\left\{x\in \cup_j I_j^*:\left(\sum_{\ell=0}^\infty|W_{\ell} b_\ell(x)|^r   \right)^{\frac{1}{r}}>\frac{t}{2} \right\}\right|+\left|\left\{x\in G\setminus  \cup_j I_j^*:\left(\sum_{\ell=0}^\infty|W_{\ell} b_\ell(x)|^r   \right)^{\frac{1}{r}}>\frac{t}{2} \right\}\right|\\
      &\leq \left|\left\{x\in G:x\in \cup_j I_j^* \right\}\right|+\left|\left\{x\in G\setminus  \cup_j I_j^*:\left(\sum_{\ell=0}^\infty|W_{\ell} b_\ell(x)|^r   \right)^{\frac{1}{r}}>\frac{t}{2} \right\}\right|.
      \end{align*} Since $$  \left|\left\{x\in G:x\in \cup_j I_j^* \right\}\right| \leq \sum_{j}|I_j^*|,  $$ we have
      \begin{align*}  & \left|\left\{x\in G:\left(\sum_{\ell=0}^\infty|W_{\ell} b_\ell(x)|^2   \right)^{\frac{1}{2}}>\frac{t}{2} \right\}\right|\\
      &\leq\sum_{j}|I_j^*|+\left|\left\{x\in G\setminus  \cup_j I_j^*:\left(\sum_{\ell=0}^\infty|W_{\ell} b_\ell(x)|^2   \right)^{\frac{1}{2}}>\frac{t}{2} \right\}\right|\\
      &=K\sum_{j}|I_j|+\left|\left\{x\in G\setminus  \cup_j I_j^*:\left(\sum_{\ell=0}^\infty|W_{\ell} b_\ell(x)|^2   \right)^{\frac{1}{2}}>\frac{t}{2} \right\}\right|\\
      &\leq \frac{CK}{t}\Vert f\Vert_{L^1(G,\ell^r)}+\left|\left\{x\in G\setminus  \cup_j I_j^*:\left(\sum_{\ell=0}^\infty|W_{\ell} b_\ell(x)|^2   \right)^{\frac{1}{2}}>\frac{t}{2} \right\}\right|.
  \end{align*} Observe that  the Chebyshev inequality implies
  \begin{align*}
     & \left|\left\{x\in G\setminus  \cup_j I_j^*:\left(\sum_{\ell=0}^\infty|W_{\ell} b_\ell(x)|^r  \right)^{\frac{1}{r}}>\frac{t}{2} \right\}\right|\\
     &\leq\frac{2}{t}\int\limits_{ G\setminus  \cup_j I_j^*} \left(\sum_{\ell=0}^\infty|W_{\ell} b_\ell(x)|^r   \right)^{\frac{1}{r}}dx\\
     &=\frac{2}{t}\int\limits_{ G\setminus  \cup_j I_j^*} \left(\sum_{\ell=0}^\infty\left|\left(W_{\ell}\left(\sum_{k} b_{\ell,k}\right)   \right)(x)\right|^r\right)^{\frac{1}{r}}dx\\
     &=\frac{2}{t}\int\limits_{ G\setminus  \cup_j I_j^*} \Vert\{(W_{\ell}(\sum_{k} b_{\ell,k}) (x)\}_{\ell=0}^\infty\Vert_{\ell^r(\mathbb{N}_0)} dx\\
     &=\frac{2}{t}\int\limits_{ G\setminus  \cup_j I_j^*} \Vert\{\sum_{k} (W_{\ell}b_{\ell,k})(x) \}_{\ell=0}^\infty\Vert_{\ell^r(\mathbb{N}_0)} dx\\
     &\leq \frac{2}{t}\sum_{k}\int\limits_{ G\setminus  \cup_j I_j^*} \left(\sum_{\ell=0}^\infty\left|\left(W_{\ell}b_{\ell,k}   \right)(x)\right|^r\right)^{\frac{1}{r}}dx.
  \end{align*}
Now, if $\kappa_\ell$ is the right convolution Calder\'on-Zygmund kernel of     $W_{\ell} ,$ (see Remark \ref{theremarkof}), and by using that $\int_{I_k}b_{k,\ell}(y)dy=0,$ we have that
\begin{align*}
    \left(\sum_{\ell=0}^\infty\left|\left(W_{\ell}b_{\ell,k}   \right)(x)\right|^r\right)^{\frac{1}{r}}&=\left(\sum_{\ell=0}^\infty\left|b_{\ell,k}\ast \kappa_{\ell}(x)   \right|^r\right)^{\frac{1}{r}}\\
    &=\left(\sum_{\ell=0}^\infty\left| \int\limits_{I_k}\kappa_\ell(y^{-1}x)b_{\ell,k}(y)dy-\kappa_{\ell}(x)\int\limits_{I_k}b_{\ell,k}(y)dy \right|^r\right)^{\frac{1}{r}}\\
       &=\left(\sum_{\ell=0}^\infty\left| \int\limits_{I_k}(\kappa_\ell(y^{-1}x)-\kappa_{\ell}(x))b_{\ell,k}(y)dy \right|^r\right)^{\frac{1}{r}}.
\end{align*} Now, we will proceed as follows. By using that
$
    |b_{\ell,k}(y)|^r\leq \sum_{\ell'=0}^{\infty}|b_{\ell',k}(y)|^r,
$ we have, by an application of the Minkowski integral inequality,
\begin{align*}
    &\left(\sum_{\ell=0}^\infty\left|\left(W_{\ell}b_{\ell,k}   \right)(x)\right|^r\right)^{\frac{1}{r}}= \left(\sum_{\ell=0}^\infty\left| \int\limits_{I_k}(\kappa_\ell(y^{-1}x)-\kappa_{\ell}(x))b_{\ell,k}(y)dy \right|^r\right)^{\frac{1}{r}}\\
    &\leq \int\limits_{I_k}\left(   \sum_{\ell=0}^\infty |\kappa_\ell(y^{-1}x)-\kappa_{\ell}(x)|^r|b_{\ell,k}(y)|^r\right)^{\frac{1}{r}}dy\\
    &\leq  \int\limits_{I_k}\left(   \sum_{\ell'=0}^{\infty}|b_{\ell',k}(y)|^r\right)^{\frac{1}{r}}\left(   \sum_{\ell=0}^\infty |\kappa_\ell(xy^{-1})-\kappa_{\ell}(x)|^r\right)^{\frac{1}{r}}dy.
\end{align*} 
Consequently, we deduce,
\begin{align*}
   & \frac{2}{t}\sum_{k}\int\limits_{ G\setminus  \cup_j I_j^*} \left(\sum_{\ell=0}^\infty\left|\left(W_{\ell}b_{\ell,k}   \right)(x)\right|^r\right)^{\frac{1}{r}}dx\\
   &\leq \frac{2}{t} \sum_{k}\int\limits_{ G\setminus  \cup_j I_j^*} \int\limits_{I_k}\left(   \sum_{\ell'=0}^{\infty}|b_{\ell',k}(y)|^r\right)^{\frac{1}{r}}\left(   \sum_{\ell=0}^\infty |\kappa_\ell(y^{-1}x)-\kappa_{\ell}(x)|^r\right)^{\frac{1}{r}}dy                   dx\\
   &= \frac{2}{t} \sum_{k} \int\limits_{I_k}  \int\limits_{ G\setminus  \cup_j I_j^*} \left(   \sum_{\ell'=0}^{\infty}|b_{\ell',k}(y)|^r\right)^{\frac{1}{r}}\left(   \sum_{\ell=0}^\infty |\kappa_\ell(y^{-1}x)-\kappa_{\ell}(x)|^r\right)^{\frac{1}{r}}dx                  dy\\
    &= \frac{2}{t} \sum_{k} \int\limits_{I_k}  \left(   \sum_{\ell'=0}^{\infty}|b_{\ell',k}(y)|^r\right)^{\frac{1}{r}}  \int\limits_{ G\setminus  \cup_j I_j^*} \left(   \sum_{\ell=0}^\infty |\kappa_\ell(y^{-1}x)-\kappa_{\ell}(x)|^r\right)^{\frac{1}{r}}dxdy.
\end{align*}
By following \cite[Page 17]{CardonaRuzhanskyBesovSpaces},  for $x\in G\setminus \cup_jI_j^*,$  and  $y\in I_{k},$ we have  $4c|y|=4c\times \textnormal{dist}(y,e_G) \lesssim  \textnormal{dist}(\partial I_{k}^*,\partial I_{k})  \leq|x|.$ So,  $$\{x\in G: x\in G\setminus \cup_jI_j^*\}\subset\{x\in G:\textnormal{ for all } z\in {I_k},\,\,\, 4c|z|\leq |x| \}.$$ Now, from the estimate \eqref{ThLiWeak} in Remark \ref{theremarkof}, we deduce 
\begin{align*}
  & \int\limits_{ G\setminus  \cup_j I_j^*} \left(   \sum_{\ell=0}^\infty |\kappa_\ell(y^{-1}x)-\kappa_{\ell}(x)|^r\right)^{\frac{1}{r}}dx \leq \int\limits_{ G\setminus  \cup_j I_j^*} \sum_{\ell=0}^\infty |\kappa_\ell(y^{-1}x)-\kappa_{\ell}(x)|dx\\
   &\leq  \sum_{\ell=0}^\infty \int\limits_{ G\setminus  \cup_j I_j^*}|\kappa_\ell(y^{-1}x)-\kappa_{\ell}(x)|dx\\\
   &\leq   \sum_{\ell=0}^\infty \int\limits_{|x| >4c|y|}|2^{-\ell Q}\kappa_\ell(2^{-\ell }\cdot y^{-1}x)-2^{-\ell Q}\kappa_{\ell}(2^{-\ell }\cdot x)|dx\\
   &\lesssim  \sum_{\ell=0} 2^{-\ell \varepsilon_0}=O(1).
\end{align*}Thus, we have proved that
\begin{align*}
      &\left|\left\{x\in G:\left(\sum_{\ell=0}^\infty|W_{\ell} b_\ell(x)|^r   \right)^{\frac{1}{r}}>\frac{t}{2} \right\}\right|\lesssim  \frac{2}{t} \sum_{k} \int\limits_{I_k}  \left(   \sum_{\ell'=0}^{\infty}|b_{\ell',k}(y)|^r\right)^{\frac{1}{r}}dy \\
      &=\frac{2}{t}  \int\limits_{\cup_{k}I_{k}}  \left(   \sum_{\ell'=0}^{\infty}|b_{\ell'}(y)|^r\right)^{\frac{1}{r}}dy \\
      &\lesssim \frac{1}{t}\Vert \{f_\ell\}\Vert_{L^1(\ell^r)}. 
\end{align*} Thus, the proof of the weak (1,1) inequality is complete and we have that \begin{equation}\label{weakvectorvalued2}
    W:L^{1}(G,\ell^r(\mathbb{N}_0))\rightarrow L^{1,\infty}(G,\ell^r(\mathbb{N}_0)),\,\,\,1<r<\infty,
\end{equation} admits a bounded extension. The proof of Lemma \ref{Lemma2} is complete.
\end{proof}
Now, in view of Lemmas \ref{Lemma1} and \ref{Lemma2},
and the duality argument in Remark \ref{RemarkInterpolation}, we have proved Theorem  \ref{HMTTL}. 
\end{proof}

\bibliographystyle{amsplain}

\end{document}